\documentclass[a4paper,11pt,reqno]{amsart}

\usepackage{MyStyle}

\title[]{Some stability properties of Hamiltonian Poisson integrators}
\author{Oscar Cosserat}
\address{Göttingen Mathematisches Institut,
Georg-August-Universität Göttingen,
Office 021, Hauptgebaüde,
Bunsenstrasse 3-5,
37073 Göttingen - Germany.}
\email{oscar.cosserat@mathematik.uni-geottingen.de}

\date{}

\begin{document}

\maketitle

\begin{abstract}
Hamiltonian Poisson integrators are Poisson integrators that admit a modified Hamiltonian. In this article, we illustrate the importance of the existence of a modified Hamiltonian for Poisson integrators in the context of integrable and non-integrable systems. Examples of Hamiltonian systems are provided by Lotka-Volterra dynamics; in order to investigate stability properties of Hamiltonian Poisson integrators on non-integrable systems, we exhibit a non-integrable $5$-dimensional Lotka-Volterra system and pursue numerical investigations of it. 
\end{abstract}

\tableofcontents

\textbf{Acknowledgments.} We are thankful to Pol Vanhaecke, Sigrid Leyendecker and Rui Loja Fernandes for fruitful discussions. We are also thankful to Antoine Falaize for his help on the implementation part of this work. The author acknowledges the RTG 2491 from the Deutsche Forschungsgemeinschaft for partial funding.

\section{Introduction}

About two centuries ago, Hamilton introduced what is nowadays called Hamiltonian mechanics. There, he wrote: "The development of this view [...] appears to me to open in mechanics and astronomy an entirely new field of research" \cite[p. 33]{Hamilton1833}. A fortunate match happened in the late fifties, when engineers observed how efficient it was in their numerical computations of Hamiltonian dynamics to preserve this Hamiltonian structure \cite{deVogelaere1956, Newmark1959}. Structure preserving numerical methods spread thereafter in various fields such as molecular dynamics \cite{Verlet1967}, solid mechanics \cite{Simo1995}, celestial mechanics \cite{laskar2004}, cosmology \cite{Croton2006}, plasma physics \cite{Kraus2017} and two-dimensional turbulence \cite{Cifani2022}. Applied to such dynamical systems, those methods have been proven to behave remarkably well with respect to the symmetries and conservation laws of those Hamiltonian systems (conservation of the energy, angular momentum, incompressibility...) and to possess surprising stability properties.

Somehow, at the time where those successful simulations started to be ruled out, no mathematical reason was known for those stability properties. A second fortunate match occured there. Indeed, a link has been made in the nineties between the perturbation theory of Hamiltonian systems and the stability of numerical methods preserving the Hamiltonian structure \cite{sanzserna1994, Benettin1994, Lubich1997, Reich1999}. As a consequence, KAM theory \cite{Nekhoroshev1977} has been used to deliver an explanation of the long run stability of symplectic integrators when they are applied to completely integrable Hamiltonian systems \cite{Shang1999}. The main tool for this task is called Backward Error Analysis. It uses the fact that the discrete trajectories of the numerical scheme follows the Hamiltonian flow of a modified Hamiltonian in order to study the iterates of the numerical method. An account of the theory of backward error analysis has been written in \cite[Chap. IX]{Hairer2006} and globalized to manifolds in \cite{Hansen2011}.

From now, let us mention two goals of the present article. 

The first goal is the generalisation of the aforementioned KAM estimates to the context of completely integrable systems in Poisson geometry. Indeed, there exists a notion of completely integrable system in Poisson geometry that generalizes the one of symplectic geometry. It is therefore natural to ask if a good notion of Poisson integrator also admits long run estimates when applied to such a system. The answer is yes: this notion of Poisson integrator has been called Hamiltonian Poisson integrator in \cite{Oscar2022}. Hamiltonian Poisson integrators are precisely Poisson integrators that admit a modified Hamiltonian. 

The second goal is to study the stability of Hamiltonian Poisson integrators when they are applied to approximate periodic orbits around elliptic singularities. In the literature about the backward error analysis of symplectic integrators, long run estimates are derived by assuming the existence of a compact subset on which long run iterations remain (see \cite[Cor. 6]{Lubich1997}, \cite[Sec. 5]{Reich1999}, \cite[Thm 3.3, eq. (3.8)]{Hansen2011}). This assumption -- the existence of a compact set that is preserved by long run iterations -- is true and proven in the case of symplectic integrators applied to integrable systems using KAM theory. In this context, this delivers a mathematical proof of the following phenomenon: the discrete trajectories of a symplectic integrator oscillate with exponentially small oscillations around quasi-periodic trajectories. Numerous numerical evidences also suggest that this assumption is true in a more general context than integrable systems, but no mathematical proof is known to us. Therefore, the second goal of this article is to investigate the existence of quasi-periodic oscillations around periodic orbits using -- instead of this assumption -- the existence of a modified Hamiltonian and invariant theory of symplectic and Poisson geometry. These stability properties of Hamiltonian Poisson integrators are observed to hold even with a large time-step.

\emph{Outline of the article}: We start in Section \ref{sec:integrable} by recalling the concept of integrable systems in Poisson geometry. This section also contains the first result of this article: Theorem \ref{thm:estimates} is the generalisation from symplectic integrators to Hamiltonian Poisson integrators of long run estimates on integrable systems using KAM theory. The section \ref{sec:model} introduces Lotka-Volterra dynamical systems. There, we give some examples of integrable systems in Poisson geometry and run Hamiltonian Poisson integrators on an integrable system to provide numerical illustrations of the theorem \ref{thm:estimates}. The section \ref{sec:stability_ell} investigates the behavior of Hamiltonian Poisson integrators on periodic orbits coming from ellipticity properties of the Hamiltonian vector field. The theorem \ref{thm:stability_periodic} provides a stability property of Hamiltonian Poisson integrators around elliptic singularities. The toy example of the Euler symplectic scheme applied to the harmonic oscillator is examined. To finish, we exhibit a non-trivial example of a Hamiltonian system where the theorem \ref{thm:estimates} does not apply. We use invariant theory of Hamiltonian systems to pursue a theoretical study of this dynamics one hand, and to comment the behavior of a Hamiltonian Poisson integrator on this system on the other hand.

\section{Hamiltonian Poisson integrators on integrable systems}\label{sec:integrable}

Let $M$ be a smooth manifold, $H \in \func{M}$ and $\{.,.\}$ Poisson brackets on $M$. We refer to \cite{Marle1987} for definitions and properties of Hamiltonian systems in Poisson geometry.

\subsection{Hamiltonian Poisson integrators}

Let $I$ be any real open interval containing $0$. We define an integrator as a family $\phi = (\phi_{\epsilon})_{\epsilon \in I}$ of smooth diffeomorphisms of $M$ such that for any $x \in M$, the map $
\begin{array}{ccc}
I &\to& M\\
\epsilon &\mapsto& \phi_{\epsilon}(x)
\end{array}
$
is smooth.

\begin{remark}[Backward error analysis]\label{rem:BEO}
A key point is the existence of a formal series $\tilde X = \sum\limits_{i \in \N} \frac{\epsilon^i}{i!} X_i $ with coefficients in $\mathcal{X}(M)$ such that the autonomous flow of $\tilde X$ at time $\epsilon$ equals $\phi_\epsilon$: for any $f \in \func{M}$, for any $k \in \N$,
\begin{equation}
f \circ \phi_\epsilon = f \circ \varphi^{\tilde X^k}_\epsilon + \smallO{k} \in \s{\func{M}}{\epsilon},
\end{equation}
where $\tilde X^k = \sum\limits_{i=0}^k \frac{\epsilon^i}{i!} X_i$. $\tilde X$ is called the \emph{modified vector field} of the integrator $\phi$. The use of $\tilde X$ to study the properties of $\phi$ is precisely called the \emph{backward error analysis}.
\end{remark}

The integrator $\phi$ is said to be of order $k$ for the vector field $X \in \mathcal{X}(M)$ if $\phi$ agrees with the flow $\varphi^X$ up to order $k$: in equation, for any smooth function $f \in \func{M}$,
\begin{equation}
f \circ \phi_\epsilon = f \circ \varphi^X_\epsilon + \smallO{\epsilon^k}.
\end{equation}

\begin{remark}
If $\phi$ is of order $k$ for $X$, then: $\tilde X = X + \frac{\epsilon^k}{k!}X_k + \ldots$, meaning that
\begin{equation}
\tilde X - X = \bigO{\epsilon^k} \in \s{\mathcal{X}(M)}{\epsilon}.
\end{equation}
\end{remark}

In the sequel, "an integrator of order $k$ for the Hamiltonian vector field of $H$" will be shortened as "an integrator of order $k$ for $H$". The following definition is out of \cite[Def. 2]{Oscar2022}.

\begin{definition}[Hamiltonian Poisson integrator]
A Hamiltonian Poisson integrator of order $k$ for $H$ is an integrator $\phi$ of order $k$ for $H$ such that $\phi$ is the time-dependent flow of a Hamiltonian $(H_t)_{t \in I}$. In equation, for any smooth function $f \in \func{M}$ and for any $\epsilon \in I$,
\begin{equation}
\frac{\partial (f \circ \phi_\epsilon) }{\partial \epsilon} = \{f, H_\epsilon \} \circ \phi_\epsilon \in \func{M}.
\end{equation}
\end{definition}

From now, let us denote by $\phi$ a Hamiltonian Poisson integrator for $H$ at order $k$. As a consequence of the Magnus formula of \cite[Sec. 1]{oscar}, there exists a modified Hamiltonian for $\phi$. More precisely, this condition guarantees that the formal vector field of the remark \ref{rem:BEO} is a Hamiltonian vector field for the Poisson structure on $M$: there exists a sequence $(H_i)_{i \in \N} \in \func{M}^{\N}$ such that, by denoting $\tilde H = \sum\limits_{i \in \N} \frac{\epsilon^i}{i!} H_i \in \s{\func{M}}{\epsilon}$, the Hamiltonian vector field of $\tilde H$ is the modified vector field $\tilde X$ of $\phi$. Up to a factor $\epsilon$, $\tilde H$ can\footnote{Note that $\tilde H$ is not unique: any Casimir might be added to $\tilde H$ without changing its Hamiltonian vector field.} be constructed precisely through the Magnus series (see \cite[Sec. 1]{oscar}) of the time-dependent Hamiltonian $(H_t)_{t \in I}$:
\begin{equation}
\tilde H = \frac{1}{\epsilon} \mathcal{M}\left( (H_t)_{t \in I} \right) = H + \epsilon \frac{\partial H}{\partial t}_{| t = 0} + \epsilon^2 \ldots \in \s{\func{M}}{\epsilon}. 
\end{equation}

An important class of examples of Hamiltonian Poisson integrators is given by symplectic integrators. In the general case of Poisson geometry, the Poisson structure foliates the whole space into symplectic leaves. The following proposition ensures the preservation of the symplectic leaves by any Hamiltonian Poisson integrator (\cite{Oscar2022}).
\begin{proposition}\label{prop:Ham_Pois_leaf}
A Hamiltonian Poisson integrator stays on a symplectic leaf along iterations. It preserves consequently any Casimir.
\end{proposition}

\subsection{Completely integrable systems in Poisson geometry}

We recall briefly integrable Poisson systems in Poisson geometry after \cite[chap. 12]{LPV2012}. Let us denote by $N$ the dimension of $M$ and $2r$ the rank of $\pi$, i.e. the dimension of any symplectic leaf of maximum dimension.

\begin{definition}[Liouville Integrable system]
Let $F = (F_1, \ldots, F_s)$ be an $s$-tuple of smooth functions on $M$ and suppose that
\begin{itemize}
\item $F$ is independent on a dense open subset,
\item $F$ is in involution,
\item $r+s = N$.
\end{itemize}
Then, $(M, \pi, F)$ is called a Liouville integrable system of dimension $N$ and rank $2r$.
\end{definition}

$F \colon M \to \R{s}$ is called the momentum map. Let us set $M_{(r)}$ to be the open subset where $\pi$ is of rank $2r$ and $U_F$ to be the dense open subset where $F$ is independent.

\begin{theorem}[Action-angle variables]
Let $(M, \pi, F)$ be a Liouville integrable system. Let $m \in U_F \cap M_{(r)}$ and suppose that $F_m = F^{-1}( \{ F(m) \})$ is compact. Let $\mathcal{F}_m$ be any connected component of $F_m$. Then, there exists $\R{}$-valued smooth functions $(a_1, \ldots, a_{N-r})$ and $\R{}/\mathbb{Z}$-valued smooth functions $(\theta_1, \ldots, \theta_r)$ defined in a neighborhood $U$ of $\mathcal{F}_m$ such that
\begin{itemize}
\item the map defined by $(\theta_1, \dots, a_1, \dots, a_{N-r})$ is a diffeomeorphism between $U$ and $\mathbb{T}^r \times B^{N-r}$,
\item the Poisson structure $\pi$ is written
\begin{equation}
\pi = \sum\limits_{i=1}^r \frac{\partial }{\partial \theta_i} \wedge \frac{\partial }{\partial a_i}
\end{equation}
in those coordinates,
\item The functions $F_1, \ldots, F_{N-r}$ depend on the functions $a_1, \ldots, a_{N-r}$ only.
\end{itemize}
\end{theorem}

The functions $\theta_1, \ldots, \theta_r$ are called \textit{angle coordinates}, the functions $a_1, \ldots, a_r$ are called \textit{action coordinates} and the functions $a_{r+1}, \ldots, a_{N-r}$ are called \textit{transverse coordinates}.

\subsection{Exponentially long times estimates for the discrete trajectory}

This section is a generalization of \cite[Sec. X]{Hairer2006} and the proof follows the same pattern. We prove that a Hamiltonian Poisson integrator of order $k$ applied to a Liouville integrable system preserves the momentum map at order $k$ over exponentially long time.

Let $\phi$ be a Hamiltonian Poisson integrator of order $k$ for $H$. From now on, we assume $M$, $H$, $\phi$ and the Poisson brackets to be real analytic. We also assume $H$ to be completely integrable, meaning that there exists a momentum map $F$ such that its first component is $H$ and $(M, \pi, F)$ is Liouville integrable. The momentum map is assumed to be real analytic as well.

\begin{theorem}[Long run estimates]\label{thm:estimates}
Let $x_0 \in M$. If $F^{-1}( \{ F(x_0) \})$ is compact, then there exists a set $\Set \subset \R{r}$ of measure zero for the Lebesgue measure such that for any $a \in Im(F) \setminus \Set \subset \R{r}$, the following holds. There exist positive constants $c_0$, $c$, $C$, $\epsilon_0$ and $\mu_0$ such that for all $0 \leq  \epsilon \leq \epsilon_0$ and for all $\mu \leq \mu_0$, every discrete trajectory starting with
\begin{equation}\label{eq:thmestimates1}
\| F(x_0) - a \| \leq c_0 \epsilon^{2 \mu}
\end{equation}
satisfies
\begin{equation}\label{eq:thmestimates2}
n \epsilon \leq \exp( c \epsilon^{-\frac{\mu}{r+1}} ) \Rightarrow \| F(x_n) - F(x_0) \| \leq C \epsilon^k,
\end{equation}
where $(x_n)_{n \in \N} \in M^{\N}$ are the iterates of $\phi_\epsilon$.
\end{theorem}

\begin{proof}
Let $(\theta,p)$ be action-angle coordinates at $x_0$. We set 
\begin{equation}
\Set = \{ a \in \R{s}, \quad (\frac{\partial H}{\partial a_1}(a), \ldots,  \frac{\partial H}{\partial a_r}(a)) \in \R{r} \text{ is not strongly non resonant } \}.
\end{equation}
First, we prove that the Lebesgue measure of $\Set$ is zero. The set of strongly non-resonant frequencies is of full measure and since the differential of $H$ is non-degenerate almost everywhere, the map
\begin{equation}
(\frac{\partial H}{\partial a_1}, \ldots,  \frac{\partial H}{\partial a_r}) \colon \R{s} \to \R{r}
\end{equation}
sends sets of zero measure onto sets of zero measure. Consequently, $\Set$ is of measure zero in any action-angle variables.

Let $a \in \R{s} \setminus \Set$ and let us set $\omega = (\frac{\partial H}{\partial a_1}(a), \ldots,  \frac{\partial H}{\partial a_r}(a)) \in \R{r}$. By the definition of strong non resonance, there exists $\nu>0$ and $\gamma>0$ such that
\begin{equation}
\forall \kappa \in \mathbb{Z}^r \setminus \{0 \}, \quad | <\kappa, \omega>| \geq \gamma | \kappa |^{-\nu}.
\end{equation}
Since a Hamiltonian Poisson integrator restricted to a symplectic leaf is a symplectic integrator, we apply Theorem 4.7 of \cite[Sec. X.5]{Hairer2006}: there exist positive constant $c_0$, $c$, $C$ and $\epsilon_0$ such that the following holds. For all $0 \leq \epsilon \leq \epsilon_0$ and for all $\mu \leq \min(k,\nu + r + 1)$, every $x_0$ such that $\| F (x_0) - a \| \leq c_0 e^{2 \mu}$ satisfies
\begin{equation}
n \epsilon \leq \exp(c \epsilon^{-\frac{\mu}{\nu + r +1}}) \Rightarrow \| F(x_n) - F(x_0) \| \leq C \epsilon^k.
\end{equation}
Now, we get rid of the constant $\nu$ coming from the non resonance of $a$. We set $\mu_0 = \min(k, r+1)$. Since $\mu_0 \leq \min(k,\nu + r + 1)$ and $\exp(c \epsilon^{-\frac{\mu}{\nu + r +1}}) \leq \exp(c \epsilon^{-\frac{\mu}{r +1}})$, this proves the theorem.
\end{proof}

\begin{remark}
Although formulated on a manifold, this result is local. First, the constants $c_0$, $c$, $C$ and $\epsilon_0$ depend on the symplectic leaf of $x_0$. Furthermore, the derivation of analysis estimates relies on action-angle coordinates, and such coordinates exist in general only in a neighborhood of a Liouville torus. Several tools might be required to globalize this result: see \cite{Miranda2010} for the globalisation of action-angle coordinates in Poisson geometry and \cite{Nekhoroshev1977} for more global estimates on nearly-integrable systems.
\end{remark}

The theorem \ref{thm:estimates} is a stability result for the long run iterates of a Hamiltonian Poisson integrator when applied to a Liouville integrable system. The iterates of any Hamiltonian Poisson integrator oscillate on the long run around Liouville tori. The way we stated this theorem had the lightest notations. However, this statement can be sharpened for free using the proposition \ref{prop:Ham_Pois_leaf}: since a Hamiltonian Poisson integrator remains on a symplectic leaf, the $r-s$ last components of the momentum map $F$ are exactly preserved, providing an improvement of the equations \eqref{eq:thmestimates1} and \eqref{eq:thmestimates2}. Indeed, two foliations are involved here: the symplectic foliation of the Poisson structure and the foliation by Liouville tori coming from the integrability of the Hamiltonian system. We come back to these foliations at the end of the section \ref{sec:LV_integrable}.

\section{The Lotka-Volterra model}\label{sec:model}

In this section, we introduce the Lotka-Volterra model. We explain the meaning of the equations, provide their Hamiltonian structure when it exists, and use this class of ODEs to give an example of an integrable systems in Poisson geometry in which one can benchmark our Hamiltonian Poisson integrator and illustrate the stability theorem \ref{thm:estimates}.

\subsection{The evolution equation and related Hamiltonian structures}\label{sec:evo_eq}

The differential equation
\begin{equation}\label{eq:LV}
\dot{x}_j = \varepsilon_j x_j + \sum\limits_{k=1}^{n} a_{jk} x_j x_k, \quad 1 \leq j \leq N,
\end{equation}
has been introduced in \cite{Volterra1931} to describe $n$ biological species competing. In this model, $x_j$ is a dynamical quantity standing for the number of individuals of the species $j$, $a_{jk}$ are coefficients describing the action of the species $k$ on the evolution of the species $j$ and the $\varepsilon_j$ are environment parameters: $\varepsilon_j >0$ means that the species $j$ increases thanks to the environment while $\varepsilon_j<0$ indicates that the species $j$ decreases because of it. $\varepsilon_j = 0$ means that without any interaction with other species, the population of the species $j$ remains constant.

\begin{remark}[Applications of Lotka-Volterra systems]\label{rk:applications}
Lotka-Volterra dynamical systems have been studied in many fields to model various phenomena, e.g. population dynamics \cite{Anish2024}, biology \cite{Ogundero2025}, control theory for biology \cite{Jones2020} or ecology \cite{Dopson2024}.
\end{remark}

In the rest of this article, the equation \eqref{eq:LV} has to be understood as restricted to the positive open quadrant 
$$\mathcal{Q}^+ =\{ (x_i)_{1 \leq i \leq N} \in \R{N}, \quad \forall 1 \leq i \leq N, \quad x_i >0 \}.$$

The Hamiltonian structure of Lotka-Volterra systems is well-known (\cite{Plank1995,Fernandes1998}) and has been used to study symmetries and integrability of such dynamical systems (\cite{vanhaecke2016}). The condition for an ordinary differential equation of the form \eqref{eq:LV} to be Hamiltonian is spelled in the following proposition.

\begin{proposition}\label{prop:Ham_structure}
Let us assume that the matrix $A = (a_{ij})_{1 \leq i, j \leq N}$ is anti-symmetric. If there exists $q \in \R{N}$ such that 
\begin{equation}\label{eq:fixed}
\forall 1 \leq k \leq N, \quad \varepsilon_j + a_{jk}q_k = 0,
\end{equation}
then \eqref{eq:LV} is Hamiltonian for the cluster Poisson structure 
\begin{equation}\label{eq:cluster}
\{ x_i, x_j \} = a_{ij} x_i x_j
\end{equation}
and the Hamiltonian
\begin{equation}\label{eq:H}
H(x) = \sum\limits_{j=1}^N (x_j - q_j \log x_j).
\end{equation}
\end{proposition}

\begin{remark}
Physically, the anti-symmetric matrix means that the interaction between two species is reciprocal: $a_{ij} = -a_{ji}$. The existence of $q$ means the existence of a fixed point in the dynamics, in other words: an equilibrium in the population dynamics.
\end{remark}

\begin{remark}
The Hamiltonian given by Equation \eqref{eq:H} is convex and has compact levels. The trajectory $x(t)$ is therefore defined at all time.
\end{remark}

In the rest of this article, we will always assume the equations \eqref{eq:LV} to be Hamiltonian.

\begin{proposition}\label{prop:cas}
Let $v \in \ker A$. Then, the map
\begin{equation}
\begin{array}{ccc}
\mathcal{Q}^+ &\to& \R{} \\
x &\mapsto& \Pi_{k=1}^N x_k^{v_k}
\end{array}
\end{equation}
is a Casimir.
\end{proposition}

\subsection{Hamiltonian Poisson integrators for Lotka-Volterra systems}\label{sec:Pois_int}

A birealisation for the cluster Poisson structure \eqref{eq:cluster} has been given in \cite{oscar}. We recall it here:
\begin{equation}\label{eq:bireal_cluster}
\left\{
\begin{array}{ccccccc}
\alpha &\colon& \R{2N} \to \R{N} &\colon& (x,p) &\mapsto&  \left( e^{-\frac{1}{2} \sum\limits_{i=1}^N a_{ij}x_i p_i }x_j \right)_{1 \leq j \leq N}\\
\beta &\colon& \R{2N} \to \R{N} &\colon& (x,p) &\mapsto& \hspace{-1 cm} \alpha(x,-p)
\end{array}
\right.
\end{equation}
This birealisation provides Hamiltonian Poisson integrators at any order for any Hamiltonian system on cluster Poisson structures. We stick to the order $1$: for any Hamiltonian $H \in \func{\R{n}}$, the birealisation \eqref{eq:bireal_cluster} gives the integrator
\begin{equation}\label{eq:Pois_Ham_1}
\begin{array}{rl}
x_n &= \alpha(y_n, \epsilon \cdot \nabla_{y_n} H)\\
x_{n+1} &= \beta(y_n, \epsilon \cdot \nabla_{y_n} H) 
\end{array},
\end{equation}
where $x_{n+1}$ is the $n+1$-th iteration computed out of $x_n$ through the intermediary point $y_n$. Note that this method is implicit: at each iteration, $y_n$ is computed out of $x_n$ through a nonlinear equation. $\epsilon$ is the time-step, chosen small enough so that $y_n$ is well-defined.

\begin{remark}
Let us denote by $\phi_\epsilon$ the integrator of the equation \eqref{eq:Pois_Ham_1}. Using symplectic groupoid theory \cite{Weinstein1987}, one can prove that any singularity of the Hamiltonian vector field of $H$ is a fixed point for $\phi_\epsilon$. In equation:
\begin{equation}
\forall x \in \R{N}, \left[ X_H(x) = 0 \Rightarrow \phi_\epsilon(x) = x \right].
\end{equation}
This will matter in the section \ref{sec:stability_ell}, in particular for the theorem \ref{thm:stability_periodic} and its consequences, where we investigate the numerical behavior of $\phi_\epsilon$.
\end{remark}

The following proposition is proven in \cite{oscar}.

\begin{proposition}
The integrator \eqref{eq:Pois_Ham_1} is a Hamiltonian Poisson integrator of order $1$ for $H$. Its time-dependent Hamiltonian is given by
\begin{equation}
H_t \colon 
\begin{array}{cccc}
\R{N} & \to &  \R{} &\\
x & \mapsto & H(y), & \text{ where } y \text{ is implicitly defined by } x = \alpha (y, t \nabla_y H).
\end{array}
\end{equation}
\end{proposition}

The construction of higher order Hamiltonian Poisson integrators through birealisations are explained in \cite{laurent2025}. The goal of this article is to illustrate the benefits of the existence of a modified Hamiltonian. Since this study is more qualitative than quantitative, all the simulations of Hamiltonian Poisson integrators we show are at order $1$.

\subsection{Examples of integrable systems and numerical simulations}\label{sec:LV_integrable}

Now, we illustrate numerically the long run estimates of the theorem \ref{thm:estimates} on a completely integrable Lotka-Volterra system. Liouville integrable Lotka-Volterra systems with $N=2$ or $3$ degrees of freedom are easy to construct; in order to obtain Liouville integrable systems of higher dimensions, we concatenate several Liouville integrable Lotka-Volterra systems.

\begin{figure}[h!]

\begin{subfigure}[b]{\textwidth}
\includegraphics[width = 0.8\textwidth]{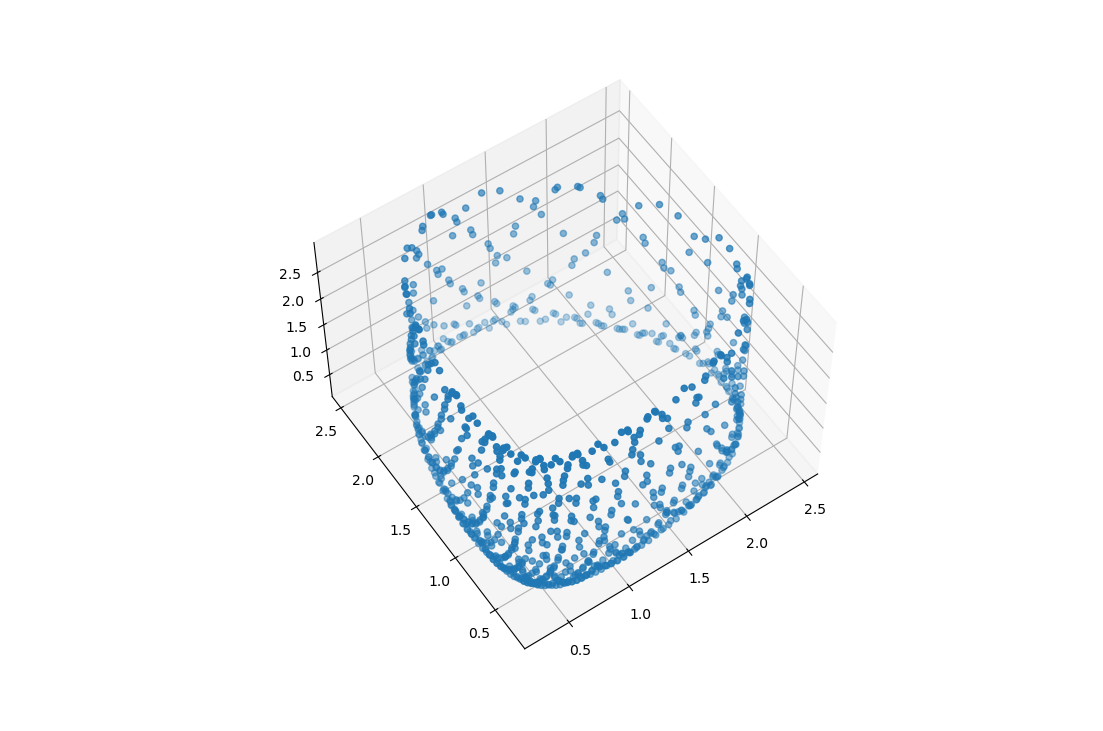}
\caption*{Quasi-periodicity: the $10^{2}$ first iterations of the $3$ first coordinates}
\end{subfigure}
\\
\begin{subfigure}[b]{0.475\textwidth}
\includegraphics[width = 0.9\textwidth]{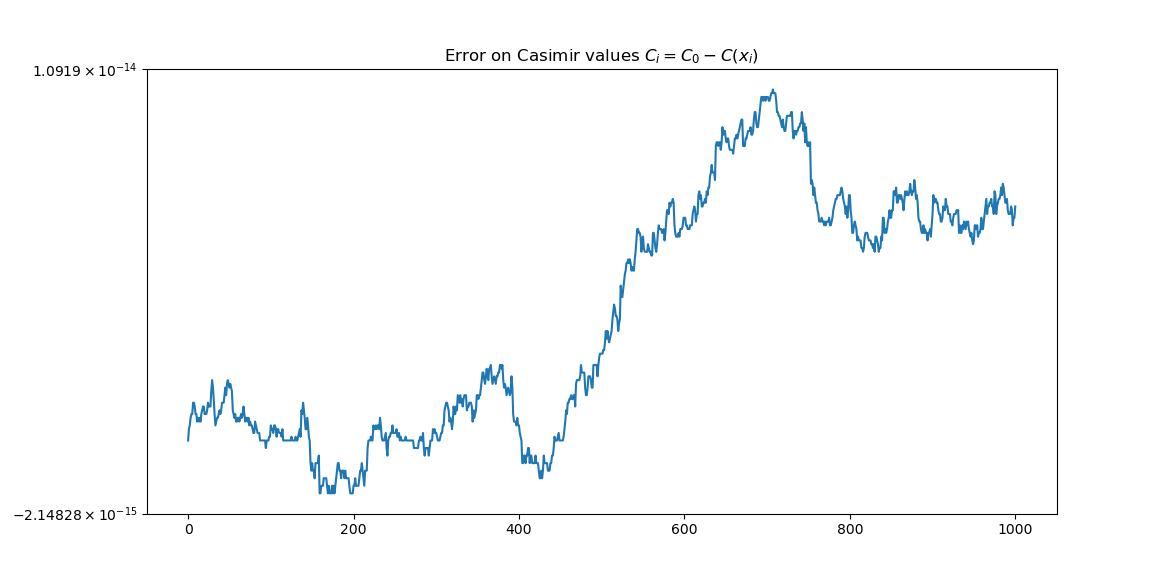}
\caption*{The Casimir discrepancy}
\end{subfigure}
\hfill
\begin{subfigure}[b]{0.475\textwidth}
\includegraphics[width = 0.9\textwidth]{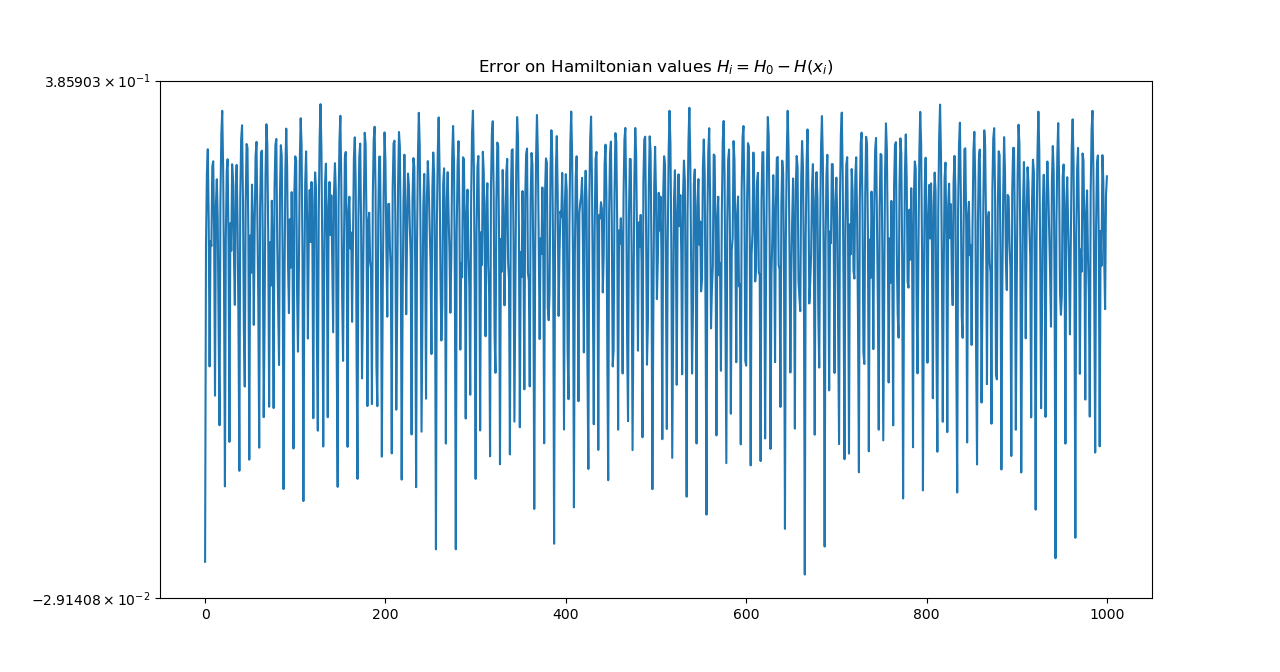}
\caption*{The Hamiltonian discrepancy}
\end{subfigure}

\caption{The $10^2$ first iterations of the Hamiltonian Poisson integrator \eqref{eq:Pois_Ham_1} applied to the system \eqref{eq:LV_5d_integrable} with time-step $\Delta t = 1$ and initial point $x_0 = (2,2,2,2,2)$}
\label{fig:Ham_Pois_1_integrable}
\end{figure}

We set the dimension $N = 5$, the interaction matrix $A = \begin{pmatrix} 
0 & 1 & 0 & 0 & 0 \\
-1 & 0 & 0 & 0 & 0 \\
0 & 0 & 0 & 1 & 1\\
0 & 0 & -1 & 0 & 0 \\
0 & 0 & -1 & 0 & 0 
\end{pmatrix}$ and $\varepsilon = \begin{pmatrix} -1 \\ 1 \\ 2 \\ -1 \\ -1 \end{pmatrix}$. The corresponding Lotka-Volterra system admits the fixed point $q = \begin{pmatrix} 1 \\ 1 \\ 1 \\ 1 \\ 1 \end{pmatrix}$ and is Hamiltonian for $H(x) = \sum\limits_{i=1}^5 (x_i - \log x_i)$. The system we look at is
\begin{equation}\label{eq:LV_5d_integrable}
\begin{array}{ccc}
x_1 &=& x_1(-1 + x_2)\\
x_2 &=& x_2(1 - x_1)\\
x_3 &=& x_3(2 + x_4 + x_5)\\
x_4 &=& x_4(-1 -x_3) \\
x_5 &=& x_5(-1 - x_3)
\end{array}
\end{equation}

It is an integrable system because there is one Casimir $C = \frac{x_4}{x_5}$ and the Hamiltonian splits into two commuting functions $H_1$ and $H_2$. Indeed, 
\begin{equation}
H = I_1 + I_2, \text{ where } 
\begin{array}{cc}
I_1 (x) =& x_1 + x_2 - \log x_1 - \log x_2\\
I_2 (x) =& x_3 + x_4 + x_5 - \log x_3 - \log x_4 - \log x_5
\end{array}
\end{equation}
and the integrability of \eqref{eq:LV_5d_integrable} comes from $\{I_1, I_2\} = 0$. The error estimates of the theorem \ref{thm:estimates} applies. We simulate the method \eqref{eq:Pois_Ham_1} and provide the results of the simulation in the figure \ref{fig:Ham_Pois_1_integrable}. There, the time-step is chosen purposely big to show the stability properties of the integrator. 

Let us conclude this section with two perspectives.

First, the fact that the Hamiltonian Poisson integrators detect faithfully Liouville tori allows to use them to investigate the integrability of Lotka-Volterra systems. To exhibit integrable Lotka-Volterra systems is a non-trivial task \cite{Bountis2021} and has been tackled with algebraic tools, see, e.g. \cite{Bountis2016}. To our knowledge, the question of the characterization of the integrability of Lotka-Volterra systems with linear terms, i.e., non-zero $\varepsilon$, remains essentially open, and we hope Hamiltonian Poisson integrators to be of some help in this task.

The second perspective deals with geometric numerical methods in control theory. In this section, we constructed numerical methods that are proved to behave well with respect to the symplectic foliations on one hand, and with respect to the tori foliations provided by integrable Hamiltonian systems on the other hand. This suggests further application in control theory, where the preservation of leaves, seen as constraints for the system, matters from a numerical point of view. The construction of numerical methods that are stable with respect to foliations is mainly open, whereas the range of potential field applications is wide, see the remark \ref{rk:applications}. We emphasize that our methods are only applicable if the conditions of the proposition \ref{prop:Ham_structure} are fulfilled.

\section{Around non-resonant elliptic singularities}\label{sec:stability_ell}

In this section, we study the behavior of Hamiltonian Poisson integrators when they are applied to compute trajectories around a non-resonant elliptic singularity of a Hamiltonian vector field. In section \ref{sec:BEO_elliptic}, we introduce useful definitions, terminology and properties concerning Hamiltonian vector fields around non-resonant elliptic singularities. In particular, we explain how those singularities provide periodic orbits and how useful is the backward error analysis in order to analyse the behavior of Hamiltonian Poisson integrators in this context. The section \ref{sec:euler_sym_elliptic} illustrates this on the toy example of the Euler symplectic scheme applied to a harmonic oscillator. In section \ref{sec:non_integrable}, we exhibit a non-integrable Hamiltonian dynamics where the theorem \ref{thm:estimates} does not apply. To finish, we pursue some numerical investigations on this non-integrable system in section \ref{sec:num_investig}, where stability properties of Hamiltonian Poisson integrators are used to investigate some chaotic behavior of the dynamics.

\subsection{Backward error analysis for the stability of Hamiltonian Poisson integrators around elliptic orbits}\label{sec:BEO_elliptic}

The idea of the next definition is the one of a smooth family of periodic orbits, all going around one same point. We were not able to exactly find this definition in the literature, even though the concept of family of periodic orbits is part of the folklore of dynamical systems \cite[chap. III]{poincare1892}, \cite[Sec. 7]{seifert1948}, \cite{moser1976}.

\begin{definition}\label{def:family}
A family of periodic orbits of a vector field $X \in \mathcal{X}(M)$ around a point $q \in M$ is a surface $\Pi \subset M$ such that:
\begin{itemize}
\item $q \in \Pi$,
\item $\Pi$ is foliated by periodic orbits around $q$ and the foliation is regular: there exists $\eta_0 >0$ such that for all $0 < \eta < \eta_0$, there exists a periodic integral curve $\gamma_\eta$ of $X$, with $\Pi \setminus \{q\} =  \bigcup_{0< \eta < \eta_0} \gamma_\eta$. 
\end{itemize}
\end{definition}

For the sequel, we introduce now a convenient terminology. Let $M$ be a Poisson manifold of dimension $N$ and $q \in M$ be on a symplectic leaf of dimension $2r$. Let $H \in \func{M}$ be a Hamiltonian. $q$ will be called a non-resonant elliptic singularity of $H$ if
\begin{itemize}
\item $q$ is a singularity for $X_H$: $X_H (q) = 0$,
\item this singularity is elliptic: the linearisation $X_H^{'}(q)$ has $2r$ purely imaginary eigenvalues $\lambda_1, -\lambda_1, \ldots, \lambda_r, -\lambda_r$,
\item and non-resonant: that the set $(\lambda_j)_{1 \leq j \leq r}$ is free over $\mathbb{Z}$, i.e.
\begin{equation}
\forall \nu \in \mathbb{Z}^{r} \left[ \sum\limits_{1 \leq j \leq r} \nu_j \lambda_j = 0 \Rightarrow  \nu = 0 \right].
\end{equation}
\end{itemize}

This definition is a natural extension on a Poisson manifold of the classical definition of a non-resonant elliptic singularity in symplectic geometry. The following statement ensures the existence of a family of periodic orbits in a neighborhood of an elliptic singularity and can be found, for instance, in \cite[Thm. 5.6.7]{Abraham1978}.

\begin{theorem}[Liapounov theorem]\label{thm:lyapou}
Let $(P, \omega)$ be a symplectic manifold of dimension $2r$ and $H$ be twice continuously differentiable. Let $q$ be a non-resonant elliptic singularity of $H$. Then, there exists $r$ families of periodic orbits of $X_H$ around $q$.
\end{theorem}

\begin{remark}\label{rk:geom_interp}
Under the assumptions of the theorem \ref{thm:lyapou}, the eigenvalues of $\frac{\partial H}{\partial x}(q)$ come by pairs $(\lambda, -\lambda) \in \C^2$. Every pair $(\lambda, -\lambda)$ is associated with a family of periodic orbits $\Pi$ in the following geometric way: the real and imaginary parts of the eigenvector of $\lambda$ are a basis of the tangent space $T_{q} \Pi$ at $q$. This will be used later on.
\end{remark}

The following theorem uses the backward error analysis to investigate the behavior of Hamiltonian Poisson integrators in a compact neighborhood of a non-resonant elliptic singularity.

\begin{theorem}\label{thm:stability_periodic}
Let $M$ be an analytic Poisson manifold, and let $H$ be an analytic smooth function. Let $q$ be a non-resonant elliptic singularity for $X_H$, $\Pi$ be the corresponding family of periodic orbits and $K$ be a compact neighborhood of $q$ in $M$. Let $(\phi_\epsilon)_{\epsilon \in I}$ a Hamiltonian Poisson integrator of order $k$ for $H$ such that 
\begin{itemize}
\item for any $x \in K$, the map
\begin{equation}
\phi(x) \colon 
\begin{array}{ccc}
I & \to & M\\
\epsilon & \mapsto & \phi_\epsilon (x)
\end{array}
\end{equation}
is analytic,
\item for all $\epsilon \in I,$ $\phi_\epsilon(q) = q$.
\end{itemize}
Then, there exists a metric $d$ on $M$, $\epsilon_0 >0,$ $(\tilde{H}_\epsilon)_{-\epsilon_0 < \epsilon < \epsilon_0 } \in \func{]- \epsilon_0, \epsilon_0[ \times M}$, $C>0$, $\beta>0$ and a measurable subset $\mathcal{D} \subset [0, \epsilon_0]$ of full Lebesgue measure such that the following holds. For every $\epsilon \in \mathcal{D}$,
\begin{enumerate}
\item \begin{equation}
\forall x \in K, \quad d\Big(\phi_\epsilon(x), \Phi_\epsilon^{X_{\tilde{H}_\epsilon}}(x) \Big) \leq C \epsilon e^{- \frac{\beta}{\epsilon}},
\end{equation}
\item $X_{\tilde{H}_\epsilon}$ has a non-resonant elliptic singularity at $q$.
\end{enumerate}
\end{theorem}

\begin{proof}
The first item is proven in \cite[Thm. 3.1., (ii)]{Hansen2011}.
The second item comes from the stability of eigenvalues under small perturbations, $\epsilon$ being picked in a set of full measure such that the perturbed eigenvalues remain non-resonant. 
\end{proof}

\begin{remark}
\cite[Sec. 3]{Mclachlan2010} already noticed the importance of resonances and eigenvalues of the linearised Hamiltonian vector field at singularities in order to study the behavior of geometric integrators for the numerical computation of periodic orbits.
\end{remark}

\begin{figure}[h]
\includegraphics[width = 0.8\textwidth]{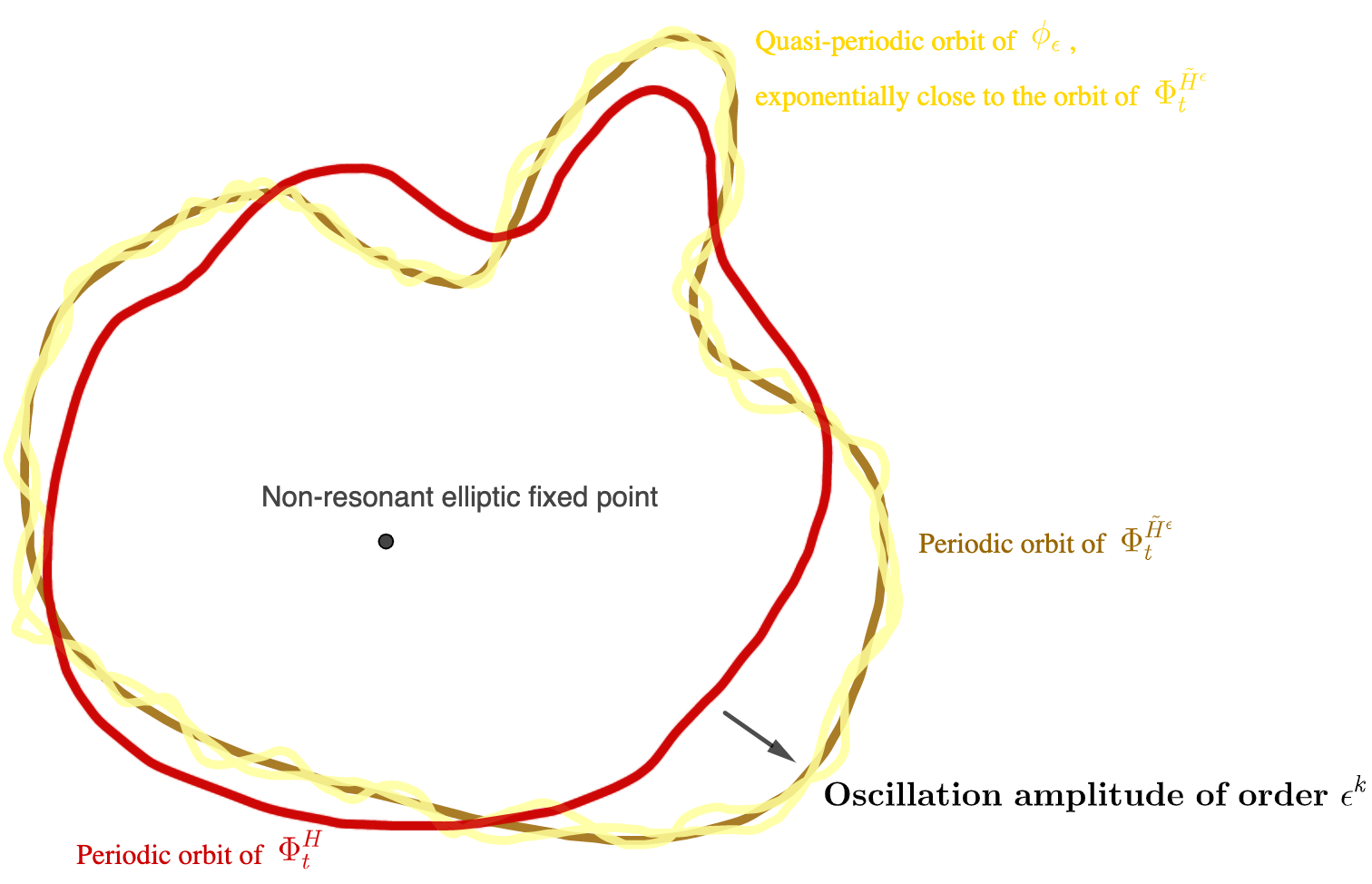}
\caption{The behavior of a Hamiltonian Poisson integrator of order $k$ applied to a Hamiltonian system admitting an elliptic non-resonant singularity on a symplectic leaf of dimension $2$}
\label{fig:elliptic_orbit}
\end{figure}

Let us explain a numerical consequence of the theorem \ref{thm:stability_periodic} on a symplectic leaf of dimension $2$. The Hamiltonian Poisson integrator oscillates with exponentially small oscillations around a periodic trajectory. This periodic trajectory is itself oscillating around the periodic orbit of the Hamiltonian $H$ with oscillation amplitudes of order $\epsilon^k$, i.e. polynomially small and controlled by the order of the method. This, in turn, delivers the fact that a Hamiltonian Poisson integrator, when applied to a Hamiltonian vector field around a non-resonant elliptic singularity, delivers quasi-periodic orbits for the Hamiltonian Poisson integrators. This is illustrated by the sketch of the figure \ref{fig:elliptic_orbit}. The perturbation behavior of Hamiltonian systems around elliptic singularities in higher dimension can be intricate, as it will be illustrated by the example provided in section \ref{sec:non_integrable}.

\subsection{A toy example: the Euler symplectic scheme for the harmonic oscillator}\label{sec:euler_sym_elliptic}

The harmonic oscillator
\begin{equation}\label{eq:harm_osc}
\begin{array}{ccc}
\dot u &=& v\\
\dot v &=& -u
\end{array}
\end{equation}
admits a (non-resonant) elliptic singularity at $0$, with eigenvalues $i$ and $-i$. The symplectic Euler method applied on the system \eqref{eq:harm_osc} reads
\begin{equation}\label{eq:euler_sym}
\begin{array}{cc}
u_{n+1} =& u_n + \epsilon v_n - \epsilon^2 u_n \\
v_{n+1} =& v_n - \epsilon u_n
\end{array}.
\end{equation}

We found numerically the maximum value of $\epsilon_0$ for the theorem \ref{thm:stability_periodic}: the method \eqref{eq:euler_sym} admits periodic orbits for any time-step $0 < \epsilon < 2$, and those orbits are necessary close to the ones of the harmonic oscillator. Numerical simulations are plotted on figure \ref{fig:euler_sym_harm_osc}.

\begin{figure*}
        \centering
        \begin{subfigure}[b]{0.475\textwidth}
            \centering
            \includegraphics[width=\textwidth]{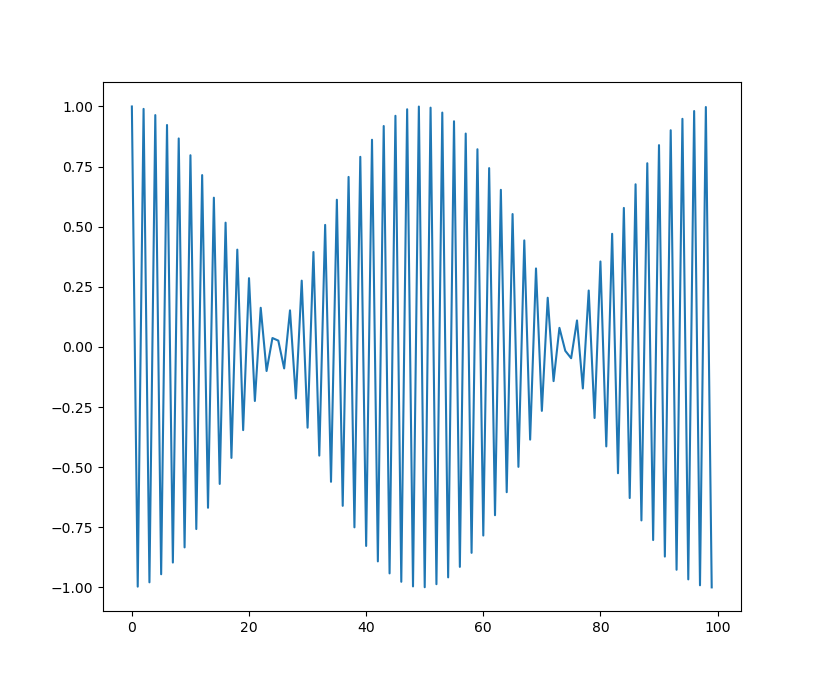}
            \caption{Coordinate $u$, $\epsilon  = 2 - 10^{-3}$}    
        \end{subfigure}
        \hfill
        \begin{subfigure}[b]{0.475\textwidth}  
            \centering 
            \includegraphics[width=\textwidth]{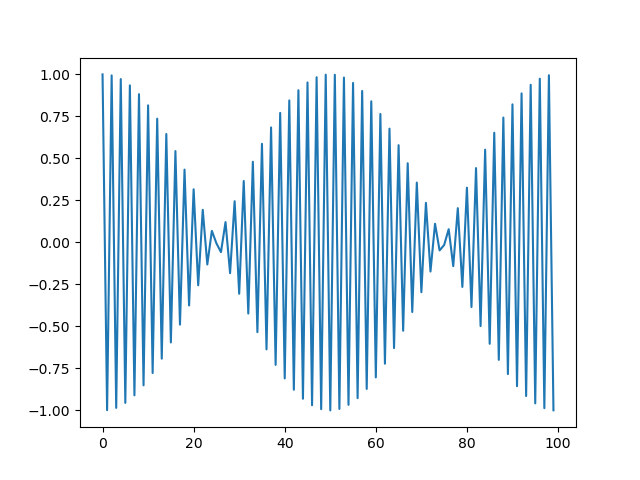}
            \caption[]%
            {Coordinate $v$, $\epsilon  = 2 - 10^{-3}$}    
        \end{subfigure}
        \vskip\baselineskip

        \begin{subfigure}[b]{0.475\textwidth}   
            \centering 
            \includegraphics[width=\textwidth]{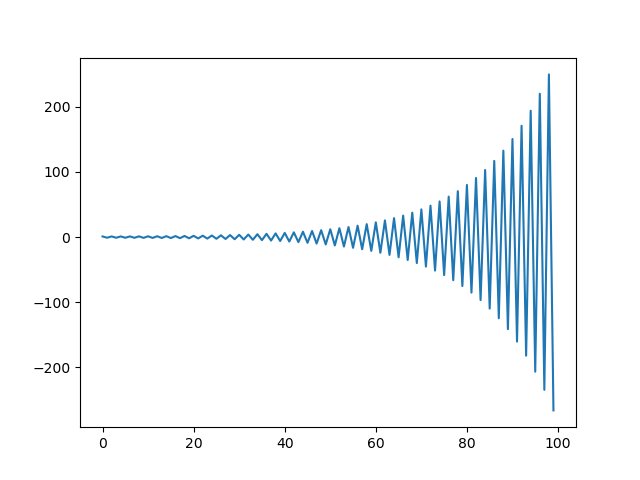}
            \caption[]%
            {Coordinate $v$, $\epsilon  = 2 + 10^{-3}$}    
        \end{subfigure}
	\hfill
        \begin{subfigure}[b]{0.475\textwidth}   
            \centering 
            \includegraphics[width=\textwidth]{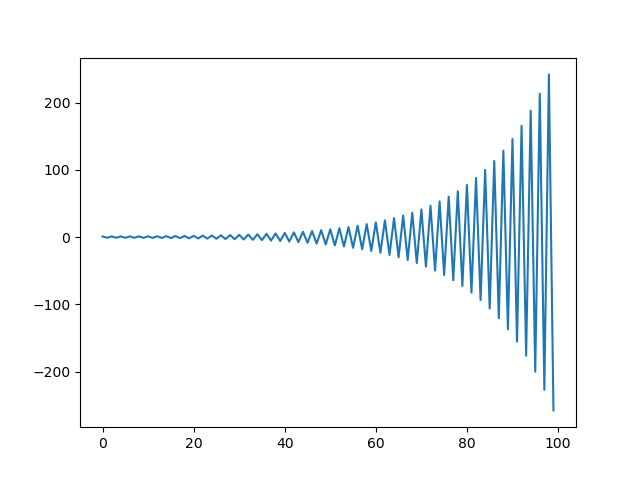}
            \caption[]%
            {Coordinate $u$, $\epsilon  = 2 + 10^{-3}$}    
        \end{subfigure}
        \caption{The first $100$ iterations of the numerical method \eqref{eq:euler_sym} with a time-step $\epsilon$} 
        \label{fig:euler_sym_harm_osc}
\end{figure*}

Now, we comment on the value $\epsilon_0 = 2$ and relate it to Backward error analysis, i.e. to properties of a modified Hamiltonian for the method \eqref{eq:euler_sym}.

\begin{proposition}\label{prop:euler_sym_elliptic}
Let $\tilde X \in \s{\mathcal{\R{2}}}{\epsilon}$ be the modified vector field of \eqref{eq:euler_sym} and $\tilde{X}^1 = X + \epsilon X_1$ its first order truncation.
$\tilde{X}^1$ has an elliptic singularity at $0$ if and only if $0 < \epsilon < 2$.
\end{proposition}

\begin{proof}
We compute a modified Hamiltonian of the scheme \eqref{eq:euler_sym} up to order $1$ in $\epsilon$. Let us denote it by $\tilde H^1 = H + \epsilon H_1$. We prove that the Hamiltonian vector field of $\tilde H^1$ has an elliptic singularity at $0$ if and only if $0 < \epsilon < 2$. 
We first compute the time-dependent Hamiltonian $H_t$ such that the time-dependent flow of $H_t$ at time $\epsilon$ is the integrator $\phi$ with time-step $\epsilon$ given by the equation \eqref{eq:euler_sym}. Let us set
\begin{equation}
\phi_\epsilon \colon \begin{pmatrix} u \\ v \end{pmatrix} \mapsto \begin{pmatrix} u + \epsilon v - \epsilon^2 u \\ v - \epsilon u \end{pmatrix}.
\end{equation}
Then, its inverse is
\begin{equation}
\phi_\epsilon^{-1} \colon \begin{pmatrix} u \\ v \end{pmatrix} \mapsto \begin{pmatrix} u - \epsilon v \\ \epsilon u + (1- \epsilon^2) v \end{pmatrix}
\end{equation}
and the time-dependent Hamiltonian vector field of $H_t$ is
\begin{equation}
X_{H_t} = \frac{\partial \phi_t}{\partial t} \circ \phi_t^{-1} \colon \begin{pmatrix} u \\ v \end{pmatrix} \mapsto \begin{pmatrix} - t u + (1 - 3 t^2) v\\
- u + t v  \end{pmatrix},
\end{equation}
out of which one obtains $H_t = \frac{u^2}{2} + \frac{1 - 3 t^2}{2} v^2 - t uv$.
 Then, the Magnus formula of $H_t$ up to order $2$ is
 \begin{equation}
 \mathcal{M}\left( (H_t)_t \right)_\epsilon = \epsilon H_0 + \frac{\epsilon^2}{2} \frac{\partial H_t}{\partial t}_{| t = 0} + \bigO{\epsilon^3} = \frac{\epsilon}{2} (u^2 + v^2) - \frac{\epsilon^2}{2} uv + \bigO{\epsilon^3}. 
 \end{equation}
By setting $\hat H_\epsilon (u,v) = \frac{\epsilon}{2} (u^2 + v^2) - \frac{\epsilon^2}{2} uv$ and $X_{\hat H_\epsilon}$ its Hamiltonian vector field, the characteristic polynomial of its linearisation at $0$ is
\begin{equation}
P(\lambda) = \det( \lambda Id - \frac{\partial X_{\hat H_\epsilon}}{\partial (u,v)}(0,0)) = \lambda^2 + \epsilon^2 (1 - \frac{\epsilon^2}{4}).
\end{equation} 
$P$ has two conjugated imaginary roots if and only if $0 < \epsilon < 2$: in that case, $P(\lambda) = (\lambda + i \epsilon \sqrt{1 - \frac{\epsilon^2}{4}})(\lambda - i \epsilon \sqrt{1 - \frac{\epsilon^2}{4}})$. This concludes the proof.
\end{proof}

\begin{remark}
We observed numerically that the discrete trajectories are also quasi-periodic for the edge time-step value $\epsilon = 2$. This is not explained by the proposition \ref{prop:euler_sym_elliptic}. To understand what happens for $\epsilon = 2$, we may have to consider higher terms of the Magnus series, i.e. higher terms of any modified Hamiltonian of the Euler symplectic scheme.

It is also tempting to generalise the proposition \ref{prop:euler_sym_elliptic}. Let $q \in M$ and $H \in \func{M}$ such that $X_H(q)$ is elliptic non-resonant. Let $\Pi$ a family of periodic orbits around $q$ provided by the Lyapunov theorem and $x_0 \in \Pi$. Can we make use of the first terms of the Magnus series to compute the biggest time-step for which a Hamiltonian Poisson integrator will remain quasi-periodic --at least for long time-- while applied at $x_0$ ?
\end{remark}

\subsection{A non-integrable Hamiltonian system}\label{sec:non_integrable}

Most of the Hamiltonian systems are not completely integrable; actually, most of the Hamiltonian systems on a symplectic manifold do not have any other first integrals than the Hamiltonian \cite[Thm. 6]{Robinson1970}. It is therefore interesting to exhibit a non-integrable Hamiltonian system in order to observe how a given integrator behaves there. In this section, we use \cite[sec. IV. 8]{Kozlov1996} and \cite[Sec. 8]{Fernandes1998} to prove the existence of a non-integrable Lotka-Volterra dynamical system in order to deliver -- in the next section -- numerical investigations of it. The idea is to construct a dynamical system that has, locally around an elliptic singularity, too many different periodic orbits to be integrable.

\begin{remark}
In Hamiltonian mechanics, the idea of studying periodic orbits to prove the nonexistence of first integrals is classical, see, e.g., \cite[Appendix 9]{arnold}.
\end{remark}

For any $\delta \in ]-1, 1[$, we introduce the system
\begin{equation}\label{eq:syst_delta}
\begin{array}{cc}
\dot x_1 =& x_1 (2 - x_2 - x_3)\\
\dot x_2 =& x_2 (-2 + x_1 + x_3 + \delta x_4)\\
\dot x_3 =& x_3 (\delta + x_1 - x_2)\\
\dot x_4 =& x_4 (-1 + x_5 - \delta x_2)\\
\dot x_5 =& x_5 (1 - x_4)
\end{array}.
\end{equation}
This system is a Lotka-Volterra dynamical system with 
\begin{equation}
\varepsilon = \begin{pmatrix} 2 \\ -2 \\ \delta \\ -1 \\ 1 \end{pmatrix} \text{ and } A = 
\begin{pmatrix}
0 & -1 & -1 & 0 & 0\\
1 & 0 & 1 & \delta & 0\\
1 & -1 & 0 & 0 & 0\\
0 & -\delta & 0 & 0 & 1\\
0 & 0 & 0 & -1 & 0
\end{pmatrix}.
\end{equation}
It admits the fixed point $q_\delta = \begin{pmatrix} 1 - \delta \\ 1 \\ 1 \\ 1 \\ 1 + \delta
\end{pmatrix}$ and therefore is Hamiltonian for 
$$H_\delta (x) = \sum\limits_{i=1}^5 (x_i - \log x_i) + \delta \log(\frac{x_1}{x_5}).$$ $u = \begin{pmatrix} 1 \\ 1 \\ -1 \\ 0 \\ \delta \end{pmatrix} \in \ker A$ and as a consequence of the proposition \ref{prop:cas}, the cluster Poisson structure associated to the matrix $A$ admits
\begin{equation}
C_\delta \colon x \in \mathcal{Q}^+ \mapsto \frac{x_1 x_2 x_5^{\delta}}{ x_3} \in \R{}
\end{equation}
 as a Casimir.

Let us emphasize that the perturbation parameter $\delta$ is not only in the Hamiltonian $H$ but also in the Poisson structure. We make a first observation.

\begin{proposition}\label{prop:integrable_orbits}
For $\delta=0$, the dynamical system \eqref{eq:syst_delta} is completely integrable. It admits two families of periodic orbits around $q = \begin{pmatrix} 1 \\ 1 \\ 1 \\ 1 \\ 1 \end{pmatrix}$, namely:
\begin{itemize}
\item $\Pi^{1,0} = \{ \begin{pmatrix} 1 \\ 1 \\ 1 \\ x_4 \\ x_5 \end{pmatrix} \in \R{5}, x_4, x_5 \in \R{} \}$
\item and $\Pi^{2,0} = \{ \begin{pmatrix} x_1 \\ x_2 \\ x_3 \\ 1 \\ 1 \end{pmatrix} \in \R{5}, x_1, x_2, x_3 \in \R{}, x_1 x_3 = x_2 \}$.
\end{itemize}
\end{proposition}

\begin{proof}
Indeed, the dynamical system \eqref{eq:syst_delta} is the concatenation of two integrable systems:
\begin{equation}
\begin{array}{cc}
\dot x_1 =& x_1 (2 - x_2 - x_3)\\
\dot x_2 =& x_2 (-2 + x_1 + x_3)\\
\dot x_3 =& x_3 (x_1 - x_2)
\end{array}
\end{equation}
with the two first integrals $H_1(x) = \sum_{i=1}^3 (x_i - \log x_i)$ and $I(x) = \frac{x_1 x_2}{x_3}$, and
\begin{equation}
\begin{array}{cc}
\dot x_4 =& x_4 (-1 + x_5)\\
\dot x_5 =& x_5 (1 - x_4)
\end{array}
\end{equation}
with the first integral $H_2(x) = \sum_{i=1}^2 (x_i - \log x_i)$.
\end{proof}

\begin{remark}
In the definition \ref{def:family} of a family of periodic orbits, we assumed the surface $\Pi$ to be of compact closure, whereas $\Pi^{1,0}$ and $\Pi^{2,0}$ are not. For our purpose, all that matters happens near the singularity $q$ so that we may, in the sequel, restrict those families to a compact neighborhood of $q$ without spelling it explicitly.
\end{remark}

In order to prove the non-integrability of the system \eqref{eq:syst_delta}, we use the following intermediary result. This lemma is an easy consequence of the theorem \ref{thm:persistence}, whose details are provided in the appendix \ref{sec:criterion}.

\begin{lemma}\label{cor:duister}
For $\delta$ small enough, the families of periodic orbits given by the proposition \ref{prop:integrable_orbits} persist and provide two families of periodic orbits $\Pi^{1, \delta}$ and $\Pi^{2,\delta}$ around $q_\delta$.
\end{lemma}

We now use the Liapounov theorem to obtain two more families of periodic orbits.

\begin{lemma}\label{cor:liapou}
For any $\delta \in ]-1, 1[ \setminus \mathbb{Q} $, the dynamical system \eqref{eq:syst_delta} admits two families $\Pi^{3, \delta}$ and $\Pi^{4, \delta}$ of periodic orbits around $q_\delta$. Furthermore, there exists $\delta_0 >0$ such that for all $\delta \in ]-\delta_0, \delta_0[ \, \bigcap \, (\R{} \setminus \mathbb{Q})$, the families of periodic orbits $\Pi^{k, \delta}$, $1 \leq k \leq 4$ are distinct.
\end{lemma}

\begin{proof}
Let us set
\begin{equation}
M_\delta = \frac{\partial X_{H_\delta}}{\partial x}(q_\delta) =  \begin{pmatrix}
0 & -1 & -1 & 0 & 0 \\
1 & 0 & 1 & \delta & 0 \\
1 & -1 & 0 & 0 & 0\\
0 & - \delta & 0 & 0 & 1\\
0 & 0 & 0 & -1 - \delta & 0
\end{pmatrix}.
\end{equation}
We compute the eigenvalues of $M_\delta$. Its characteristic polynomial is
\begin{equation}
P(\lambda) = \lambda (\lambda^4 + (\delta^2 + \delta + 4) \lambda^2 + (\delta^2 + 3 \delta + 3)).
\end{equation}
Let us set $A = \delta^2 + \delta + 4$, $B = \delta^2 + 3 \delta +3$ and the discriminant $\Delta = A^2 - 2B$. An easy study shows that for any $\delta$: $\Delta >0$, $X_1 = \frac{-A-\sqrt{\Delta}}{2} < 0$ and $X_2 = \frac{-A + \sqrt{\Delta}}{2} < 0$. Therefore, $P$ factorizes as
\begin{equation}
P(\lambda) = \lambda( \lambda + i \sqrt{-X_1}) ( \lambda - i \sqrt{-X_1}) ( \lambda + i \sqrt{-X_2})( \lambda + i \sqrt{-X_2}).
\end{equation}
For $\delta \in ]1, 1[ \bigcap (\R{} \setminus \mathbb{Q})$, $q$ is an elliptic strongly non-resonant fixed point and the theorem \ref{thm:lyapou} delivers two families of periodic orbits $\widetilde \Pi^{1, \delta}$ and $\widetilde \Pi^{1, \delta}$ around $q_\delta$.

Now, we verify that for $\delta$ small enough, $\widetilde \Pi^{1, \delta}$ and $\widetilde \Pi^{2, \delta}$ are both different from the families $\Pi^{1, \delta}$ and $\Pi^{2, \delta}$. We only prove that $\widetilde \Pi^{1, \delta}$ is different from $\Pi^{1, \delta}$ and $\Pi^{2, \delta}$, the proof for $\widetilde \Pi^{2, \delta}$ being similar. The idea is to use the item (ii) of the theorem \ref{thm:persistence}. For any $\delta$, we pick $(v_1, v_2)$ a basis of the tangent space at $T_{q_\delta} \widetilde \Pi^{1, \delta}$ and we prove that as $\delta \to 0$, the vector space generated by $(v_1, v_2)$ is different from the tangent spaces $T_q \Pi^{1, 0}$ and $T_q \Pi^{2,0}$ at $\delta = 0$. Let $x = (x_1, x_2, x_3, x_4, x_5) \in \C^{5}$ and $\lambda \in \mathbb{C}$. By solving
\begin{equation}\label{eq:eigen_eq}
(M_\delta - \lambda Id)x = 0,
\end{equation}
formal computations provide a solution
\begin{equation}
x(\delta, \lambda) = 
\begin{pmatrix}
\delta \frac{1-\lambda}{\lambda^2 + 1} - \lambda +1\\
\delta + \lambda^2 + 1\\
\delta \frac{\lambda - 1}{\lambda^2 + 1} - \lambda - 1\\
-\delta \lambda\\
\delta(\delta + 1)
\end{pmatrix}
\end{equation}
and any dilation of $x$ still solves the equation \eqref{eq:eigen_eq}.
Let $\lambda_k = i \sqrt{- X_k}$, $k = 1, 2$. Using the remark \ref{rk:geom_interp}, the eigenspace $T_{q_\delta} \widetilde \Pi^{k,\delta}$ is generated by the real and imaginary parts of $x(\delta, \lambda_k)$. So $T_{q_\delta} \widetilde \Pi^{1,\delta}$ admits as generators
\begin{equation}
\begin{pmatrix}
\frac{\lambda_1^2 + 1 + \delta}{\lambda_1^2 + 1}\\
\lambda_1^2 + 1 + \delta\\
-\frac{\lambda_1^2 + 1 + \delta}{\lambda_1^2 + 1}\\
0\\
\delta (\delta + 1)
\end{pmatrix}
\text{ and }
\begin{pmatrix}
-\sqrt{- X_1}\frac{\lambda_1^2 + 1 + \delta}{\lambda_1^2 + 1}\\
0\\
-\sqrt{- X_1}\frac{\lambda_1^2 + 1 + \delta}{\lambda_1^2 + 1}\\
-\sqrt{-X_1}\\
0
\end{pmatrix}.
\end{equation}
Since $X_1 \overset{\delta \to 0}{\rightarrow} -3$, the two generators of $T_{q_\delta} \widetilde \Pi^{1,\delta}$ converge as $\delta \to 0$ to 
\begin{equation}
\begin{pmatrix}
1\\
-2\\
-1\\
0\\
0
\end{pmatrix}
\text{ and }
\begin{pmatrix}
-\sqrt{3}\\
0 \\
-\sqrt{3}\\
-\sqrt{3}\\
0
\end{pmatrix}.
\end{equation}
At $\delta = 0$, the tangent space at $q$ of the family of periodic orbits $\Pi^{1,0}$ is generated by
\begin{equation}
\begin{pmatrix}
0\\
0\\
0\\
1\\
0
\end{pmatrix}
\text{ and }
\begin{pmatrix}
0\\
0\\
0\\
0\\
1
\end{pmatrix}
\end{equation}
while the tangent space at $q$ of the family of periodic orbits $\Pi^{2,0}$ is a subspace of the vector space generated by
\begin{equation}
\begin{pmatrix}
1\\
0\\
0\\
0\\
0
\end{pmatrix},
\begin{pmatrix}
0\\
1\\
0\\
0\\
0
\end{pmatrix}
\text{ and }
\begin{pmatrix}
0\\
0\\
1\\
0\\
0
\end{pmatrix}.
\end{equation}
This makes clear that $T_{q_0} \widetilde \Pi^{1,0}$ is different from those two vector spaces and therefore, that the family of periodic orbits $\widetilde \Pi^{1,\delta}$ is different from $\Pi^{1, \delta}$ and $\Pi^{2, \delta}$ for $\delta \in \R{} \setminus \mathbb{Q}$ small enough.
\end{proof}

The following result comes out of the theory of normal forms of Hamiltonian systems and is proven in \cite[Thm. 8.2]{Fernandes1998}.

\begin{theorem}\label{thm:csq_Lyap}
If a Hamiltonian system on a symplectic manifold is completely integrable in a neighborhood of a non-resonant elliptic singular point $q$, then the only families of periodic orbits around $q$ are the ones given by the Lyapounov theorem.
\end{theorem}

The theorem \ref{thm:csq_Lyap} and the lemma \ref{cor:liapou} deliver altogether the non-integrability of the Hamiltonian system \eqref{eq:syst_delta}.

\begin{corollary}
For $\delta \in \R{} \setminus \mathbb{Q}$ small enough, the dynamical system \eqref{eq:syst_delta} is not integrable around $q_\delta$. More precisely, there exists a neighborhood of $q_\delta$ in $M$ such that on this neighborhood, the dynamical system \eqref{eq:syst_delta} does not admit any continuously differentiable first integral independent of $H$ and $C_\delta$.
\end{corollary}

\subsection{Numerical investigations of a non-integrable system}\label{sec:num_investig}

We now use the Hamiltonian Poisson integrator \eqref{eq:Pois_Ham_1} to investigate the dynamics of the system \eqref{eq:syst_delta} around its fixed point $q_\delta$. In this section, we set\footnote{We observed in practice that the assumption for $\delta$ to be irrational does not matter: the numerical results are the same as the ones obtained with $\delta = \sqrt{2}.10^{-2}$. The question on how to deal with irrational numbers in scientific computing is interesting but outside our present scope.} $\delta = 10^{-2}$. We use the results of the last section to recover the family of periodic orbits $\Pi^{k, \delta}$ and $\widetilde \Pi^{k, \delta}$ for $k = 1, 2$. All the simulations have been ruled with the time-step $\epsilon = 10^{-2}$.

Let us explain how we have chosen the initial points of the numerical simulations plotted in this section. Our task is to look for the dynamics living near the family of periodic orbits $\Pi^{k, \delta}$ and $\widetilde \Pi^{k, \delta}$ for $k = 1, 2$. Let $\Pi$ be one of those families and let $(u,v) \in \R{5} \times \R{5}$ be a basis of $T_{q_\delta} \Pi$. Let $\eta >0$ be a real parameter. In this section, the initial points of the iterations we plot are of the form
\begin{equation}
q_\delta + \eta(u+v)
\end{equation}
and we observe wether the simulations exhibit families of periodic orbits or not.

\begin{remark}
Thanks to the symplectic foliation, the context of our numerical investigation is very geometric. Numerical investigations of a transition from an integrable to a more chaotic behavior of the dynamics have been already lead in \cite[Sec. 3.2]{modin2017}, where a symplectic method on the sphere has been applied to rigid body dynamics with forcing terms.
\end{remark}

\subsubsection{Near the family of periodic orbits $\Pi^1$}

The figure \ref{fig:Pi1} shows the $5.10^2$ first iterations of the Hamiltonian Poisson integrator \eqref{eq:Pois_Ham_1} applied to the Hamiltonian system \eqref{eq:syst_delta}. Figure \ref{fig:Pi1_3first} shows the coordinates $x_1, x_2, x_3$ of the iterates of the initial point $(1-10^{-2},  1,  1, 1, 2 + 10^{-2}) = q_\delta + \eta(u+v)$ where $\eta = 1$, $u = \begin{pmatrix} 0 \\ 0 \\0 \\ 1\\0 \end{pmatrix}$ and $v = \begin{pmatrix} 0 \\ 0 \\0 \\ 0\\1 \end{pmatrix}$. The simulation plotted Figure \ref{fig:Pi1_3first} is remarkable, providing a numerical evidence of the non-integrability of the dynamics \eqref{eq:syst_delta}. The figure \ref{fig:Pi1_2last} plots the 2 last coordinates of the iterates of the three different initial points $q_\delta + \frac{i}{3}\eta(u+v)$, $1 \leq i \leq 3$, approximating dynamics near the family of periodic orbits $\Pi^{1, \delta}$. The quasi-periodicity already mentioned in section \ref{sec:BEO_elliptic} is observed in the figure \ref{fig:quasi_periodic}, obtained by zooming in the figure \ref{fig:Pi1_2last}.

\begin{figure}[h]
\centering
\begin{minipage}{.5\textwidth}
  \centering
\includegraphics[width = 0.9\textwidth]{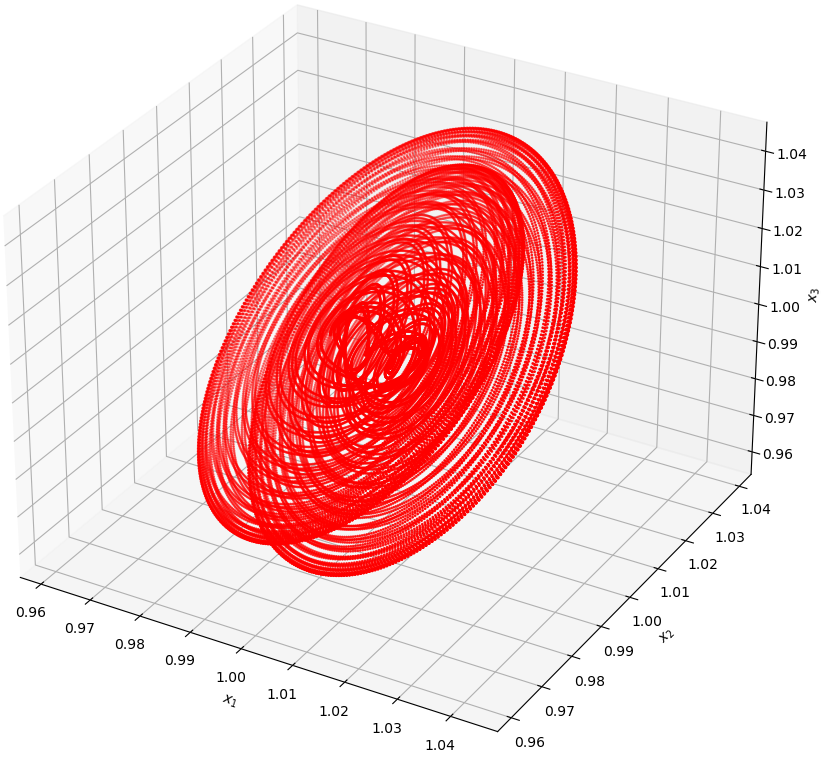}
\caption{Iterates of the coordinates $(x_1, x_2, x_3)$}
  \label{fig:Pi1_3first}
\end{minipage}%
\begin{minipage}{.5\textwidth}
  \centering
\includegraphics[width = 0.9\textwidth]{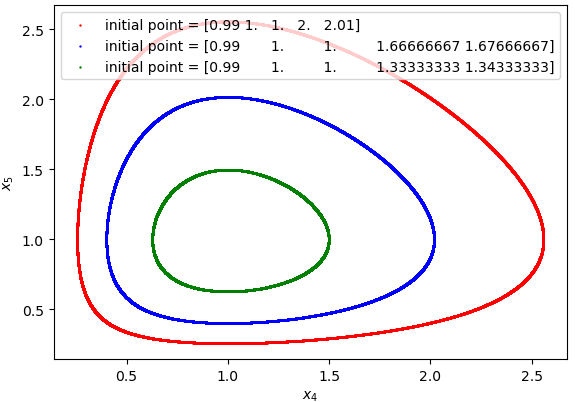}
  \caption{Iterates of the coordinates $(x_4, x_5)$}
  \label{fig:Pi1_2last}
\end{minipage}
\caption{Dynamics near the family $\Pi^{1, \delta}$}
\label{fig:Pi1}
\end{figure}

\begin{figure}[h]
\centering
\includegraphics[width = 0.6\textwidth]{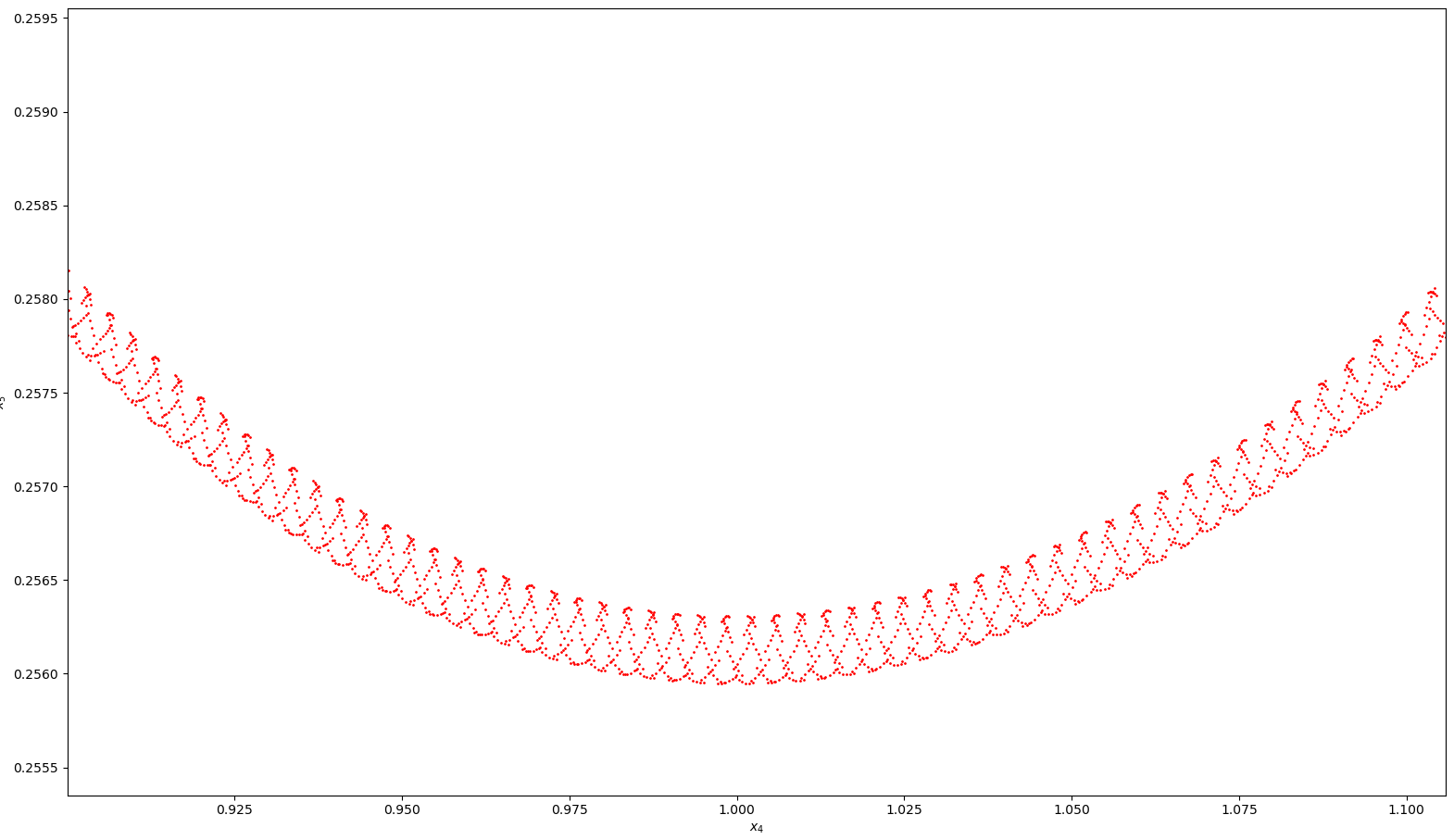}
\caption{A zoom of the figure \ref{fig:Pi1_2last} illustrating the quasi-periodicity of a Hamiltonian Poisson integrator when applied to a periodic orbit}
\label{fig:quasi_periodic}
\end{figure}

\subsubsection{Near the family of periodic orbits $\Pi^2$}

The figure \ref{fig:Pi2} shows the $10^2$ first iterations of the Hamiltonian Poisson integrator \eqref{eq:Pois_Ham_1} applied to the Hamiltonian system \eqref{eq:syst_delta}. Figure \ref{fig:Pi2_3first} shows the coordinates $x_1, x_2, x_3$ of the iterates of three initial points $q_\delta + \frac{i}{3}\eta(u+v)$, $1 \leq i \leq 3$, $\eta = 1$, $u = \begin{pmatrix} 1 \\1 \\2 \\0 \\0 \end{pmatrix}$ and $v = \begin{pmatrix} 1 \\ -1 \\ 0 \\ 0 \\ 0 \end{pmatrix}$. The figure \ref{fig:Pi2_2last} plots the 2 last coordinates of the iterates of the initial point $(\frac{7}{3} - 10^{-2},  1, \frac{7}{3}, 1, 1 + 10^{-2}) = q_\delta + \frac{2}{3}\eta(u+v)$, approximating dynamics near the family of periodic orbits $\Pi^{2, \delta}$. Figure \ref{fig:Pi2_2last} is the analog of the figure \ref{fig:Pi1_3first} for the family $\Pi^2$. They both exhibit a chaotic behavior of the dynamics. In both cases, the iterates remain on a compact set.

\begin{figure}[h]
\centering
\begin{minipage}{.5\textwidth}
  \centering
\includegraphics[width = \textwidth]{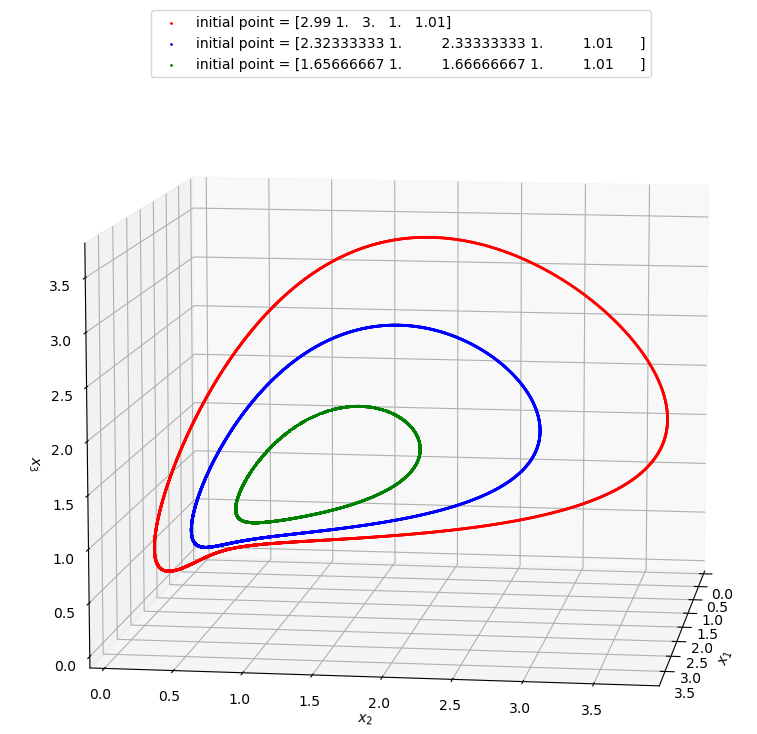}
\caption{Iterates of the coordinates $(x_1, x_2, x_3)$}
  \label{fig:Pi2_3first}
\end{minipage}%
\begin{minipage}{.5\textwidth}
  \centering
\includegraphics[width = \textwidth]{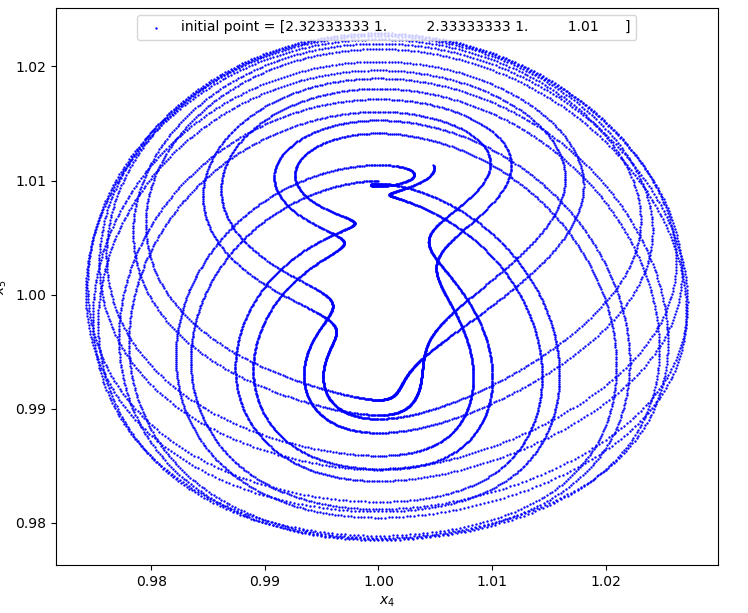}
  \caption{Iterates of the coordinates $(x_4, x_5)$}
  \label{fig:Pi2_2last}
\end{minipage}
\caption{Periodic orbits near the family $\Pi^{2, \delta}$}
\label{fig:Pi2}
\end{figure}

\subsubsection{Near the families of periodic orbits $\widetilde \Pi^1$ and $\widetilde \Pi^2$}

The figure \ref{fig:tildePi1} shows the $10^2$ first iterations of the Hamiltonian Poisson integrator \eqref{eq:Pois_Ham_1} applied to the Hamiltonian system \eqref{eq:syst_delta}. The figures \ref{fig:tilde_Pi1_3first} and \ref{fig:tilde_Pi1_2last} show, respectively, the coordinates $(x_1, x_2, x_3)$ and the coordinates $(x_4, x_5)$ of the iterates of three initial points $q_\delta + \frac{i}{3} \eta (u+v)$, $1 \leq i \leq 3$, $\eta = 10^{-1}$, $u = \begin{pmatrix} 1 \\ -2 \\ -1 \\ 0 \\ 0 \end{pmatrix}$ and $v = \begin{pmatrix} 1 \\ 0 \\1 \\ 1 \\0 \end{pmatrix}$, approximating periodic orbits near the family of periodic orbits $\widetilde \Pi^{1, \delta}$.

\begin{figure}[h]
\centering
\begin{minipage}{.5\textwidth}
  \centering
\includegraphics[width = \textwidth, height = 0.4\textheight]{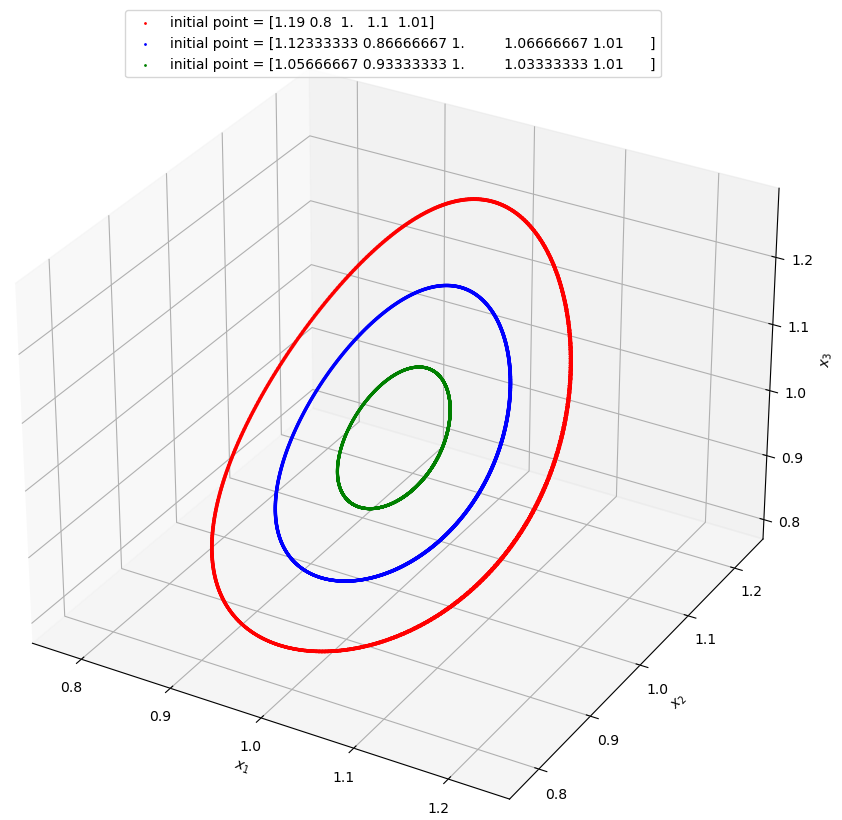}
\caption{Iterates of the coordinates $(x_1, x_2, x_3)$}
  \label{fig:tilde_Pi1_3first}
\end{minipage}%
\begin{minipage}{.5\textwidth}
  \centering
\includegraphics[width = \textwidth, height = 0.4\textheight]{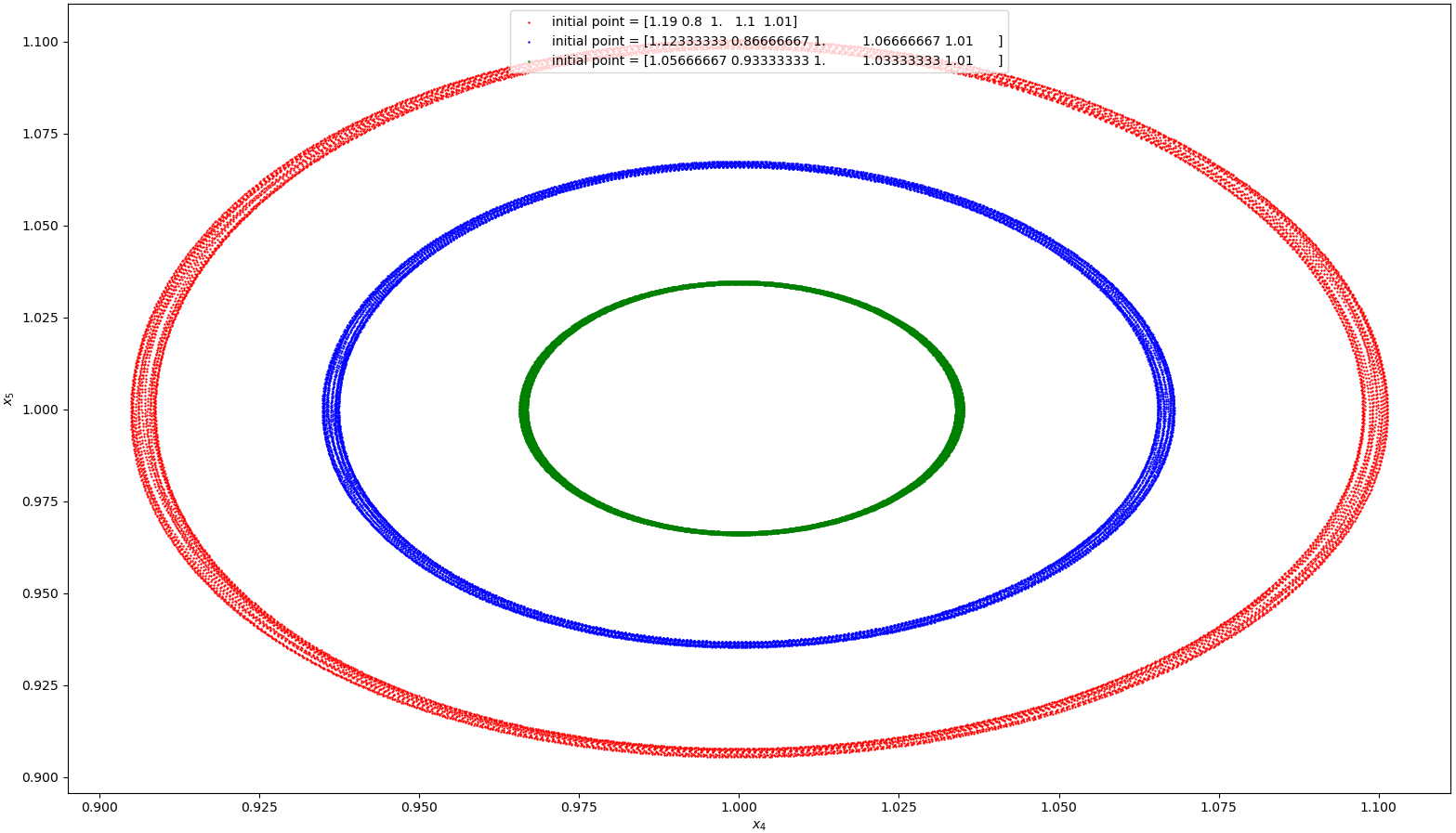}
  \caption{Iterates of the coordinates $(x_1, x_2, x_3)$}
  \label{fig:tilde_Pi1_2last}
\end{minipage}
\caption{Periodic orbits near the family $\widetilde \Pi^{1, \delta}$}
\label{fig:tildePi1}
\end{figure}

Numerical simulations for the dynamics near $\widetilde \Pi^{2, \delta}$ are of the same kind as the ones plotted on Figure \ref{fig:tildePi1}. Figure \ref{fig:tildePi1} provides numerical insights that the families $\widetilde \Pi^{k, \delta}$, $k=1, 2$, are neighbored by other families of periodic orbits. It would be certainly interesting to pursue further numerical investigations in order to precisely find the subset of $\mathcal{Q}^+$ where the transition from stable family of periodic orbits, i.e. of the type of $\widetilde \Pi^{1, \delta}$, to chaotic dynamics observed in the figures \ref{fig:Pi1_3first} and \ref{fig:Pi2_2last}, happens.

\bibliographystyle{alpha}
\bibliography{bib.bib}

\newpage

\begin{appendix}

\section{A stability criterion for a family of periodic orbits}\label{sec:criterion}

We provide here a stability criterion for a family of periodic orbits, used in the section \ref{sec:non_integrable} to obtain the families of periodic orbits $\Pi^{1, \delta}$ and $\Pi^{2, \delta}$. This theorem can be seen as a parameter version of \cite[Thm 5.6.6]{Abraham1978} and our proof sketch follows the same idea.

\begin{theorem}\label{thm:persistence}
Let $(X_\delta)_{- \delta_0 < \delta < \delta_0} \in \mathcal{X}(\R{N})^{] - \delta_0, \delta_0 [}$ be a family of smooth vector fields on a manifold $M$ depending smoothly on $\delta$. Let $(H_\delta)_{- \delta_0 < \delta < \delta_0} \in \func{\R{N} \times ] - \delta_0, \delta_0[}$ such that for all $- \delta_0 < \delta < \delta_0$, 
\begin{equation}\label{eq:first_integral}
X_\delta \cdot H_\delta = 0.
\end{equation}
Let us set $X = X_0$. We make the following assumptions.
\begin{enumerate}
\item there exists a smooth map $\begin{array}{ccc} ]- \delta_0, \delta_0[ &\to& \R{N} \\ \delta &\mapsto& q_\delta \end{array}$ such that for all $- \delta_0 < \delta < \delta_0$, $q_\delta$ is the only singularity of $X_\delta$.
\item $X$ admits a family of periodic orbits $\Pi^0$ around $q_0$.
\item $\Pi^0$ is contained in a regular hypersurface of $H_0$.
\item For any $x \in \Pi^{0} \setminus \{q_0 \}$, let $T_0 > 0$ the period of $x$. We assume
\begin{itemize}
\item $\dim \ker (T_x \Phi^X_{T_0} - Id) = 2$,
\item $X(x) \notin \Im(T_x \Phi_{T_0}^X - Id)$.
\end{itemize}
\end{enumerate}
Then, there exists $\hat \delta_0>0$ such that:
\begin{enumerate}[label = (\roman*)]
\item for all $- \hat \delta_0 < \delta < \hat \delta_0$, $X_\delta$ admits a family of perdiodic orbits $\Pi^{\delta}$ around $q_\delta$,
\item there exists a smooth map $\begin{array}{ccc}
 ] - \hat \delta_0, \hat \delta_0[ &\to & \R{N} \times \R{N} \\
  \delta & \mapsto& (u_\delta, v_\delta) \end{array}$ such that for all $- \hat \delta_0 < \delta < \hat \delta_0$, $\text{span}_{\R{}} (u_\delta, v_\delta) = T_{q_\delta} \Pi^{\delta}$.
\end{enumerate}
\end{theorem}

\begin{proof}
We set
\begin{equation}
\Theta \colon 
\begin{array}{cccccccc}
&\R{N} &\times& ]0, \infty[ &\times& ]-\delta_0, \delta_0[ &\to& \R{N}\\
&x,& &T,&& \delta & \mapsto & \Phi_{T}^{X_\delta}(x) - x
\end{array}.
\end{equation}
A zero of $\Theta$ is a periodic orbit for $X_\delta$. Let $x_0 \in \Pi^{0} \setminus \{ q_0 \}$ and $T_0 >0$ the period of $x_0$ for $X$. Let us set $Z = \ker T_{x_0} H \subset \R{N}$ and $\pi \colon \R{N} \twoheadrightarrow Z$ the normal projection. We write $e = H_0(x_0)$. Using the assumption (3), there exists $\hat \delta_0 >0$ such that for all $-\delta_0 < \delta < \delta_0$, $\Sigma_\delta = H^{-1}_\delta (e) \subset \R{N} $ is a hypersurface. Up to shrinking of $\hat \delta_0$, there exists a neighborhood $\mathcal{U}$ of $0$ in $Z$ and a smooth map $ \phi \colon \mathcal{U} \times ]- \hat \delta_0, \delta_0[ \to \R{N}$ such that $\Sigma_\delta$ is parametrized by $\phi(\cdot, \delta)$ and $\phi(0,0) = x_0$. We set
\begin{equation}
\Psi \colon 
\begin{array}{ccccccccc}
&\mathcal{U}& \times &]0, \infty[& \times &] - \hat \delta_0, \hat \delta_0[& &\to& Z\\
&x,& &T,& &\delta& &\mapsto& \pi\left( \Theta(\phi(x, \delta), T, \delta) \right)
\end{array}.
\end{equation}
Then:
\begin{align}
\frac{\partial \Psi}{\partial x}(0,T_0, 0) &= \pi\left( T_x \Phi_{T_0}^X(x_0) - Id \right)\\
\frac{\partial \Psi}{\partial T} (0, T_0, 0) &= \pi (X(x_0)) = X(x_0).
\end{align}
Out of the assumptions (4) and \eqref{eq:first_integral}, the implicit function theorem applies and provides, for $\delta > 0$ small enough, a zero of $\Psi$ and, in turn, a periodic orbit for $X_\delta$ being close to the one of $x_0$. More precisely, there exists $\hat \delta_0 >0$ and a smooth map
\begin{equation}
(x,T) \colon ]- \hat \delta_0, \hat \delta_0 [ \to \R{N} \times ]0, \infty[
\end{equation}
with $x(0)= x_0$, $T(0) = T_0$ and $\forall - \hat \delta_0 < \delta < \hat \delta_0, \; \Theta( x(\delta), T(\delta)) = 0$.

Now, we use the assumption (2). Let $K$ be a compact neighborhood of $q_0$ in $\Pi^0$. By performing the same reasoning for all $x \in K \setminus \{ q_\delta \}$ in a compact neighborhood of $q_0$ and applying assumption (1), we prove (i). The proof of (ii) comes thereafter easily by smoothness of the maps involved in the implicit function theorem.
\end{proof}

\end{appendix}

\end{document}